\documentclass[12pt,oneside,english]{amsart}
\usepackage[T1]{fontenc}
\usepackage[latin9]{inputenc}
\usepackage[letterpaper]{geometry}
\usepackage{color}
\usepackage{amsthm}
\usepackage{setspace}
\usepackage{amssymb}
\usepackage{esint}
\DeclareRobustCommand{\greektext}{%
  \fontencoding{LGR}\selectfont\def\encodingdefault{LGR}}
\DeclareRobustCommand{\textgreek}[1]{\leavevmode{\greektext #1}}
\DeclareFontEncoding{LGR}{}{}

\numberwithin{equation}{section} 
\numberwithin{figure}{section} 
\theoremstyle{plain}
 \theoremstyle{definition}
 \newtheorem*{defn*}{Definition}
\theoremstyle{plain}
\newtheorem{extthm}{Theorem}
            
  \theoremstyle{remark}
  \newtheorem*{rem*}{Remark}
\theoremstyle{plain}
\newtheorem{thm}{Theorem}
  \theoremstyle{plain}
  \newtheorem{prop}[thm]{Proposition}
  \theoremstyle{plain}
  \newtheorem{lem}[thm]{Lemma}
\newenvironment{lyxcode}
{\par\begin{list}{}{
\setlength{\rightmargin}{\leftmargin}
\setlength{\listparindent}{0pt}
\raggedright
\setlength{\itemsep}{0pt}
\setlength{\parsep}{0pt}
\normalfont\ttfamily}%
 \item[]}
{\end{list}}
  \theoremstyle{definition}
  \newtheorem{defn}[thm]{Definition}

\usepackage{babel}

\begin{document}
\global\long\def\bbZ{\mathbb{Z}}

\global\long\def\bbN{\mathbb{N}}

\global\long\def\bbX{\mathbb{X}}

\global\long\def\cF{\mathcal{F}}

\global\long\def\cB{\mathcal{B}}

\global\long\def\fT{\mathfrak{\tau}}

\global\long\def\cCT{\mathcal{C}_{\fT-1}}

\global\long\def\cC{\mathcal{C}}

\global\long\def\RR{\mathbb{R}}

\begin{singlespace}

\title{the zero-type property and mixing of Bernoulli shifts}
\end{singlespace}

\begin{singlespace}

\author{Zemer kosloff}
\end{singlespace}

\begin{singlespace}

\thanks{This research was was supported by THE ISRAEL SCIENCE FOUNDATION
grant No. 1114/08.}
\end{singlespace}
\begin{abstract}
\begin{singlespace}
We prove that every non-singular Bernoulli shift is either zero-type
or there is an equivalent invariant stationary product probability.
We also give examples of a type ${\rm III}_{1}$ Bernoulli shift and
a Markovian flow which are power weakly mixing and zero type.\end{singlespace}

\end{abstract}
\begin{singlespace}

\end{singlespace}

\maketitle

\subsection*{Introduction}

\begin{singlespace}
This article deals with the concept of zero-type for invertible non-singular
transformations $T$ of the standard probability space $\left(X,\cB,m\right)$.\textbf{ }

A non-singular transformation without an absolutely continuous invariant
probability is called zero-type, sometimes also called mixing, if
its Koopman operator is mixing, meaning that it's maximal spectral
type is a Rajchman measure. This generalises the notion of zero-type
introduced by Hajian and Kakutani \cite{H-K}.\textbf{ }\textcolor{black}{The
name mixing comes from the fact that in the case of probability preserving
transformations the classical notion of mixing is equivalent to mixing
of the Koopman operator restricted to}\textbf{\textcolor{black}{ }}$\ L_{0}^{2}\left(X,m\right)=L^{2}(X,m)\ominus\mathbb{C}$.
In the interplay between non singular ergodic theory and infinitely
divisible processes, the zero type transformations are associated
with the classical notion of mixing of infinitely divisible processes,
see \cite{RS} or \cite{Ro}. 

In the first part we prove that every non-singular Bernoulli shift
is either zero-type or there is an equivalent invariant product probability.
Thus the Hamachi shift \cite{Ham} and the type ${\rm III}_{1}$ shift
given in \cite{Kos} are examples of conservative, ergodic, zero-type
transformations which do not posses an $m$ equivalent $\sigma$-finite
invariant measure, also known as type ${\rm III}$ and zero type transformations.
For another construction of zero type, type ${\rm III}$ transformations
see \cite[Theorem 0.3]{Da1}.

A non singular transformation $T$ is weakly mixing if $S\times T$
is ergodic for every ergodic probability preserving transformation
$S$. It follows that if $T$ is a probability preserving weakly mixing
transformation then $T\times\cdots\times T$ ($k-$times) is ergodic
for every $k\in\mathbb{N}$. In general this is not true, as there
exist weakly mixing transformations such that $T\times T$ is not
ergodic. In an attempt to understand the notion of weak mixing for
non singular transformations the concept of power weakly mixing transformations
was introduced. $T$ is called \textit{power weakly mixing} if for
any $l_{1},l_{2},....,l_{k}\in\bbZ\backslash\{0\}$, \[
\left(X^{k},P^{\otimes k},T^{l_{1}}\times T^{l_{2}}\times\cdots\times T^{l_{k}}\right)\]
is an ergodic automorphism where \[
P^{\otimes k}=\underset{k-\mbox{times}}{\underbrace{P\otimes P\otimes\cdots\otimes P}}.\]
A weaker notion is that of infinite ergodic index which means that
for every $k\in\mathbb{N}$, the $k$-fold product of $T$ is an ergodic
automorphism. In \cite{AFS} it was shown that Chacon's non singular
type ${\rm III}_{\lambda}$ transformation, $0<\lambda\leq1$, is
a power weakly mixing transformation. In \cite{Da2} a construction
of a transformation which is of infinite ergodic index but not power
weakly mixing was given. Later in \cite{DP} a flow was constructed
such that all times are of infinite ergodic index. These constructions
use the cutting and stacking method which usually doesn't give a zero
type transformation.
\end{singlespace}

We give a construction of a Bernoulli shift which is power weakly
mixing and type ${\rm III}_{1}$. By the first part it is also zero
type. Finally a continuous time flow $\left\{ \phi_{t}\right\} _{t\in\RR}$
is given such that for every $a_{1},a_{2},...,a_{n}\in\RR$, \[
\phi_{a_{1}}\times\phi_{a_{2}}\times\cdots\times\phi_{a_{n}}\]
is ergodic. This flow is the time shifts of a continuous time Markov
Chain.

\subsection*{Preliminaries:}

\begin{singlespace}
Let $\left(X,\cB,P\right)$ be a probability space. An invertible
measurable transformation $T:X\to X$ is said to be \textit{non-singular}
with respect to $P$ if it preserves the $P-$null sets. i.e for every
$A\in\cB,$ $P(A)=0$ if and only if \[
P\circ T(A):=P\left(TA\right)=0.\]
A measure $m$ on $X$ is said to be $T-$\textit{invariant} if for
every\textit{ $A\in\cB$, \[
P\circ T\left(A\right)=P(A).\]
}If $T$ is non-singular, then for every $n$, $P\circ T^{n}$ is
absolutely continuous with respect to $P$. By the Radon-Nikodym theorem
there exist measurable functions $\frac{dP\circ T^{n}}{dP}\in L_{1}(X,P)_{+}$
such that for every $A\in\mathcal{B}$ , \[
P\circ T^{n}(A)=\int_{A}\left(\frac{dP\circ T^{n}}{dP}\right)\, dP.\]

Denote $\left(T^{n}\right)'(x):=\frac{dP\circ T^{n}}{dP}$.

A set $W\in\cB$ is called \textit{wandering} if $\left\{ T^{n}W\right\} _{n=-\infty}^{\infty}$
are disjoint. As in \cite[p.7]{Aar} denote by $\mathfrak{D}$ the
measurable union of all wandering sets for $T$, this set is $T-$invariant.
Its complement is denoted by $\mathfrak{C}$. 
\end{singlespace}
\begin{defn*}
\begin{singlespace}
We call $\left(X,\cB,P,T\right)$ \textit{dissipative} if $\mathfrak{D}=X$.
If $\mathfrak{C}=X$ then $\left(X,\cB,P,T\right)$ is said to be
\textit{conservative}. \end{singlespace}

\end{defn*}
\begin{singlespace}
We will use the following version of Hopf's decomposition theorem
which says that the conservative and dissipative parts can be separated
in the following way. 
\end{singlespace}
\begin{extthm}
\begin{singlespace}
\label{ext: Hopf's theorem}For every non-singular transformation
$T$ of the probability space $(X,\cB,P)$ there exists a decomposition
$X=\mathfrak{D\cup C}$ such that $\left(X,\mathcal{B},P|_{\mathfrak{D}},T\right)$
is dissipative and $\left(X,\cB,P|_{\mathfrak{C}},T\right)$ is conservative.
Furthermore \[
\sum_{k=0}^{\infty}\left(T^{-n}\right)'(x)<\infty\ a.e.\ x\in\mathfrak{D}\]
and\[
\sum_{k=0}^{\infty}\left(T^{-n}\right)'(x)=\infty\ a.e.\ x\in\mathfrak{C.}\]
\end{singlespace}

\end{extthm}

\subsection*{The Hellinger Integral and definition of the Zero Type Property}

\begin{singlespace}
The Koopman operator $U:L_{2}\left(X,\cB,P\right)$ is then defined
by\[
Uf(x)=\sqrt{T'(x)}f\circ T(x).\]
It is a unitary operator and by the chain rule for the Radon-Nikodym
derivatives for every $n\in\bbZ$,\[
U^{n}f=\sqrt{\left(T^{n}\right)'}f\circ T^{n}.\]

\end{singlespace}
\begin{defn*}
\begin{singlespace}
A transformation is \textit{Non-Singular zero-type} ( NS zero type)
if the maximal spectral type $\sigma_{T}\in\mathcal{P}\left(\mathbb{T},\mathcal{B}\left(\mathbb{T}\right)\right)$
of $U_{T}$ is a Rajchman measure. That is its Fourier coefficients
$\hat{\sigma}_{T}(n)$ tend to $0$ as $n\to\infty$. and so for every
$f\in L_{2}(X,P)$,\[
\int_{X}f\cdot U^{n}fdP\to0\ as\ n\to\infty.\]
\end{singlespace}
\end{defn*}
\begin{rem*}
\begin{singlespace}
By looking at the Koopman operators it is seen that the zero-type
property depends only on the equivalence class of $P$. \end{singlespace}

\end{rem*}
\begin{singlespace}
\end{singlespace}
\begin{rem*}
\begin{singlespace}
\label{rem: Krengel mixing and H-K}In the case when $\left(X,\cB,m,T\right)$
is a $\sigma$-finite measure preserving transformation Krengel and
Sucheston \cite{KrS} have showed that mixing of the Koopman operator
is equivalent to the Hajian and \textit{Kakutani zero-type condition}
\cite{H-K} which states that for every $A\in\cB$ with $m(A)<\infty$,\[
\lim_{n\to\infty}m\left(A\cap T^{-n}A\right)=0.\]
\end{singlespace}
\end{rem*}
\begin{defn*}
\begin{singlespace}
Let $P,Q$ be two probability measures on $X$. The \textit{Hellinger
Integral} (see \cite{Hel} or \cite{Kak}) is then defined as \[
\rho(P,Q)=\int_{X}\sqrt{\frac{dP}{dm}}\cdot\sqrt{\frac{dQ}{dm}}\, dm\]
where $m$ is any finite measure on $X$ such that $P\ll m\ ,\ Q\ll m$.
In the special case where $P\ll Q$ we can take $m=Q$ and have \[
\rho(P,Q)=\int_{X}\sqrt{\frac{dP}{dQ}}\: dQ.\]
The function $\rho(\cdot,\cdot)$ measures the amount of singularity
of $P$ with respect to $Q$. This function satisfies that for every
$P,Q\in\mathcal{P}(X)$ ,$0\leq\rho(P,Q)\leq1$. Also $\rho(P,Q)=0$
if and only if $P$ is singular with respect to $Q$. \end{singlespace}

\end{defn*}
\begin{singlespace}
The proof of the following proposition is standard. 
\end{singlespace}
\begin{prop}
\begin{singlespace}
\label{pro:NS zero type and mixing. }Let $T$ be a non-singular transformation
of the probability space $\left(X,\cB,P\right)$. The following are
equivalent.

(i) ${\displaystyle \lim_{n\to\infty}\rho\left(P,P\circ T^{n}\right)=0}$.

(ii) $\left(T^{n}\right)'\overset{P}{\longrightarrow0}.$

(iii) $\sigma_{T}$ is a Rajchman measure. \end{singlespace}

\end{prop}
\begin{singlespace}

\section{Bernoulli shifts are zero-type or mixing}
\end{singlespace}

\begin{singlespace}
Let $\bbX=\{0,1\}^{\bbZ}$ and $T$ be the left shift action on $\bbX$,
that is \[
\left(Tw\right)_{i}=w_{i+1}.\]
Denote the cylinder sets by \[
\left[b\right]_{k}^{l}=\left\{ w\in\bbX:\ \forall i=k,...,l,\ w_{i}=b_{i}\right\} .\]
 A measure $P={\displaystyle \prod_{k=-\infty}^{\infty}}P_{k}\in\mathcal{P}(\bbX)$
is called a product measure if for every $k<l$, and for every cylinder
$[b]_{k}^{l},$

\[
P\left(\left[b\right]_{k}^{l}\right)=\prod_{j=k}^{l}P_{j}\left(\left\{ b_{j}\right\} \right).\]
 We will say that a product measure $P$ is \textit{non-singular}
if the shift is non-singular with respect to $P$. 
\end{singlespace}

For two product probability measures $P,Q$ and $N\in\bbN\cup\{\infty\}$
define \begin{eqnarray*}
d_{N}\left(P,Q\right) & := & \sum_{k=-N}^{N}\left\{ \left(\sqrt{P_{k}\left(\left\{ 0\right\} \right)}-\sqrt{Q_{k}\left(\left\{ 0\right\} \right)}\right)^{2}\right..\\
 &  & \left.+\left(\sqrt{P_{k}\left(\left\{ 1\right\} \right)}-\sqrt{Q_{k}\left(\left\{ 1\right\} \right)}\right)^{2}\right\} \end{eqnarray*}
Notice that $d_{N}(P,Q)\uparrow d_{\infty}(P,Q)$ as $N\to\infty$.
Set $d(P,Q):=d_{\infty}(P,Q)$

\begin{singlespace}
The following lemma is a direct consequence of Kakutani's Theorem
on equivalence of product measures\cite{Kak}.
\end{singlespace}
\begin{lem}
\begin{singlespace}
\label{cla: Zero Type con for shifts}Let $P={\displaystyle \prod_{k=-\infty}^{\infty}P_{k}}$
be a product measure. Then
\end{singlespace}

(1) For any two product measures $P,Q$, \[
d(P,Q)\propto-\log\rho\left(P,Q\right).\]

(2) $P$ is non-singular if and only if\begin{equation}
d(P,P\circ T)<\infty.\label{eq: condition for zero type 1.}\end{equation}

\begin{singlespace}
(3) The shift $ $is NS zero-type if and only if\[
\lim_{n\to\infty}d\left(P,P\circ T^{n}\right)=\infty.\]
\end{singlespace}
\end{lem}
\begin{thm}
\begin{singlespace}
\label{thm:Let--be}Let $P={\displaystyle \prod_{k=-\infty}^{\infty}P_{k}}$
be a non-singular product measure. Then either there exists a shift
invariant $P$-equivalent probability or the shift $ $ $\left(\bbX,\cB(\bbX),P,T\right)$
is NS zero-type. Therefore a non-singular shift is either mixing in
the probability preserving sense or mixing in the non-singular sense. \end{singlespace}

\end{thm}
\begin{singlespace}
Theorem \ref{pro:NS zero type and mixing. } follows from lemmas \ref{pro:with limit}
and \ref{pro:no limit}. 
\end{singlespace}
\begin{lem}
\begin{singlespace}
\label{pro:with limit}Let $P$ be a non-singular product measure
on $\{0,1\}^{\mathbb{Z}}$ such that \[
\exists\lim_{k\to-\infty}P_{k}=(p,1-p):=\mu_{p}.\]
Denote by $Q={\displaystyle \prod_{k=-\infty}^{\infty}\mu_{p}}$.
Then if $P\perp Q$ then $ $ $\left(\bbX,\cB(\bbX),P,T\right)$ is
of NS-zero type. Else $Q$ is a $P$-equivalent shift-invariant probability
measure. \end{singlespace}
\end{lem}
\begin{proof}
\begin{singlespace}
Assume that $P\perp Q$. Then by Kakutani's theorem \[
d\left(P,Q\right)=\infty.\]
By claim \ref{cla: Zero Type con for shifts} its enough to show that
${\displaystyle \lim_{n\to\infty}d\left(P,P\circ T^{n}\right)=\infty}$. 

Let $M>0$. Since $d(P,Q)=\infty$ there exists a $N\in\mathbb{N}$
such that \[
d_{N}(P,Q)>M.\]
For every $n\in\mathbb{N}$,\begin{eqnarray*}
d\left(P,P\circ T^{n}\right) & \geq & \sum_{k=-N}^{N}\left\{ \left(\sqrt{P_{k}\left(\left\{ 0\right\} \right)}-\sqrt{P_{k-n}\left(\left\{ 0\right\} \right)}\right)^{2}\right.\\
 &  & \left.+\left(\sqrt{P_{k}\left(\left\{ 1\right\} \right)}-\sqrt{P_{k-n}\left(\left\{ 1\right\} \right)}\right)^{2}\right\} \end{eqnarray*}
Therefore since ${\displaystyle \lim_{j\to-\infty}}P_{j}=\mu_{p}$
then \begin{eqnarray*}
\liminf_{n\to\infty}d(P,P\circ T^{n}) & \geq & d_{N}(P,Q)\\
\geq & M.\end{eqnarray*}
Since $M$ is arbitrary then\[
\lim_{n\to\infty}d\left(P,P\circ T^{n}\right)=\infty.\]
\end{singlespace}
\end{proof}
\begin{lem}
\begin{singlespace}
\label{pro:no limit}Let $ $$P$ be a non-singular product measure
on $\{0,1\}^{\mathbb{Z}}$ such that \[
\liminf_{k\to-\infty}P_{k}\left(\left\{ 0\right\} \right)\neq\lim\sup_{k\to-\infty}P_{k}\left(\left\{ 0\right\} \right).\]
Then $\left(\bbX,\cB(\bbX),P,T\right)$ is NS zero-type. \end{singlespace}
\end{lem}
\begin{proof}
\begin{singlespace}
Write $q_{1}={\displaystyle \lim\inf_{k\to-\infty}P_{k}(\{0\})}$
and $q_{2}={\displaystyle \limsup_{k\to-\infty}P_{k}(\{0\})}$.

Let $M>0$. Set $\alpha=\frac{q_{2}-q_{1}}{4}.$ Define \[
A_{q_{i}}:=\left\{ n\in\mathbb{Z}:\ \left|P_{k}(\{0\})-q_{i}\right|<\alpha\right\} ,\ i=1,2.\]
Let $A_{q_{i}}^{N}=A_{q_{i}}\cap[-N,N]$. 

Choose $N$ large enough so that \[
\left|A_{q}^{N}\right|\geq\frac{M}{\alpha}\ {\rm \ and\ }\left|A_{p}^{N}\right|\geq\frac{M}{\alpha}.\]
Since $d(P,P\circ T)<\infty$ then for every $j\in\mathbb{Z}\cap[-N,N]$,\[
\lim_{n\to\infty}\left|P_{-n}(\{0\})-P_{-n+j}(\{0\})\right|=0.\]
Therefore for large enough $n\in\bbN$ either \[
\left[-N-n,N-n\right]\cap A_{q_{1}}=\emptyset\]
or \[
[-N-n,N-n]\cap A_{q_{2}}=\emptyset.\]
Therefore for large enough $n\in\mathbb{N}$, \begin{eqnarray*}
d\left(P,P\circ T^{n}\right) & \geq & d_{N}(P,Q)\\
 & \geq & \sum_{k\in A_{q_{1}}^{N}}\left(\sqrt{P_{k}\left(\left\{ 0\right\} \right)}-\sqrt{P_{k-n}\left(\left\{ 0\right\} \right)}\right)^{2}\\
 &  & +\sum_{k\in A_{q_{2}}^{N}}\left(\sqrt{P_{k}\left(\left\{ 0\right\} \right)}-\sqrt{P_{k-n}\left(\left\{ 0\right\} \right)}\right)^{2}\\
 & \geq & \min\left(\alpha\cdot\left|A_{q_{i}}^{N}\right|,\alpha\cdot\left|A_{q_{2}}^{N}\right|\right)\geq M.\end{eqnarray*}
Therefore \[
\lim_{n\to\infty}d\left(P,P\circ T^{n}\right)=\infty\]
and the shift is NS zero-type. \end{singlespace}

\end{proof}
\begin{singlespace}

\section{A zero type and power weakly mixing Bernoulli shift}
\end{singlespace}
\begin{lyxcode}
\begin{singlespace}
\end{singlespace}

\end{lyxcode}
\begin{singlespace}
In this section we construct a Bernoulli shift which is zero type
and power weak mixing. The construction is done by imposing a stronger
growth condition on the shift constructed in \cite{Kos}. 

Construction of the product measure: 

The product measure will be $P={\displaystyle \prod_{k=-\infty}^{\infty}}P_{k}$
, where \begin{equation}
\forall i\geq0,\ P_{i}\left(0\right)=P_{i}\left(1\right)=\frac{1}{2}.\label{P on N+}\end{equation}

The definition of $P_{k}$ for negative $k$'s is more complicated
as it involves an inductive procedure. 
\end{singlespace}

\subsection{The inductive definition of $P_{k}$ for negative $k's$.}

\begin{singlespace}
We will need to define inductively 5 sequences $\left\{ \lambda_{t}\right\} _{t=1}^{\infty},\left\{ n_{t}\right\} _{t=1}^{\infty}$
, $\left\{ m_{t}\right\} _{t=1}^{\infty},\ \left\{ M_{t}\right\} _{t=0}^{\infty}$
and $\left\{ N_{t}\right\} _{t=1}^{\infty}$ . The sequence $\left\{ \lambda_{t}\right\} $
is of real numbers which decreases to $1$. The other four, $\left\{ n_{t}\right\} _{t=1}^{\infty}$
, $\left\{ m_{t}\right\} _{t=1}^{\infty},\ \left\{ M_{t}\right\} _{t=0}^{\infty}$
and $\left\{ N_{t}\right\} _{t=1}^{\infty}$ are increasing sequences
of integers. \medskip{}

First choose a positive summable sequence $\left\{ \epsilon_{t}\right\} _{t=1}^{\infty}$
and set $M_{0}=1$. 

\underbar{Base of the induction}: Set $\lambda_{1}=2$ , $n_{1}=2$
, $m_{1}=4$ . Set also $N_{1}=M_{0}+n_{1}=3$ and $M_{1}=N_{1}+m_{1}=7.$
$ $

Given $\left\{ \lambda_{u},n_{u},N_{u},m_{u},M_{u}\right\} _{u=1}^{t-1}$
, we will choose the next level $\left\{ \lambda_{t},n_{t},N_{t},m_{t},M_{t}\right\} $
in the following order. First we choose $\lambda_{t}$ depending on
$M_{t-1}$ and $\epsilon_{t}$. Given $\lambda_{t}$ we will choose
$n_{t}$ and then $N_{t}$ will be defined by \[
N_{t}:=M_{t-1}+n_{t}.\]
Then given $N_{t}$ we will choose $m_{t}$ and finally set \[
M_{t}:=N_{t}+m_{t}.\]
\underbar{Choice of $\lambda_{t}$}: Set $k_{t}:=\left\lfloor \log_{2}\left(\frac{M_{t-1}}{\epsilon_{t}}\right)\right\rfloor +1$
, where $\left\lfloor x\right\rfloor $ denotes the integral part
of $x$. Then set $\lambda_{t}=e^{\frac{1}{2^{k_{t}}}}$. With this
choice of $\lambda_{t}$ we have,\begin{equation}
\lambda_{t}^{M_{t-1}}<e^{\epsilon_{t}}.\label{Constraint on lambda(t)}\end{equation}
 This choice of $\lambda_{t}$ has the property that for every $u<t$,
$\lambda_{u}=\lambda_{t}^{2^{k_{t}-k_{u}}}.$

Define \[
A_{t-1}:=\left\{ \prod_{u=1}^{t-1}\lambda_{u}^{x_{u}}:\ x_{u}\in\left[-n_{u},n_{u}\right]\right\} .\]

\underbar{Choice of $n_{t}$}: Given $\{\lambda_{u},n_{u},m_{u}\}_{u=1}^{t-1}$
and $\lambda_{t}$, the set $A_{t-1}$ is a finite subset of $\lambda_{t}^{\bbZ}$.
Choose $n_{t}$ which satisfies

\[
\lambda_{t}^{n_{t}/4}\geq\max\left\{ a^{2}:a\in A_{t-1}\right\} .\]
\underbar{Choice of $m_{t}$}: Now that $n_{t}$ is chosen we set
$N_{t}=M_{t-1}+n_{t}$ . Set \begin{equation}
tN_{t}\left(2+2^{tN_{t}}\right)=m_{t}.\label{constraint of m(t)}\end{equation}

\end{singlespace}
\begin{rem*}
\begin{singlespace}
Since $m_{t}$ satisfies \ref{constraint of m(t)} then for every
$k<t$ and $n\leq N_{t}$, \begin{equation}
\frac{m_{t}}{n}-N_{t}>2^{kN_{t}}\label{eq:The growth condition necessary for p.w.m}\end{equation}
\end{singlespace}
\end{rem*}
\begin{defn}
\begin{singlespace}
Let $\left(X,\mu,T\right)$ be a non singular automorphism such that
$\cB(X)\neq\left\{ \emptyset,X\right\} $ and let $\mathcal{F\subset B}$
be a factor algebra. Then:

(i) $\mathcal{F}$ is \textit{exhaustive} if ${\displaystyle \vee_{n=0}^{\infty}T^{n}\mathcal{F=\cB}.}$

(ii) $\mathcal{F}$ is \textit{exact} if ${\displaystyle \cap_{n=0}^{\infty}T^{n}\mathcal{F=}\left\{ \emptyset,X\right\} .}$

(iii) $T$ is \textit{a K-automorphism} if it is conservative and
admits a factor algebra $\mathcal{F\subset B}$ that is exhaustive
and exact and such that $T'$ is $\mathcal{F}$ measurable. \end{singlespace}
\end{defn}
\begin{rem*}
\begin{singlespace}
Krengel has shown in \cite[p. 153-154]{Kre} that all $K$-automorphisms
are ergodic. See also \cite[Proposition 4.8(a)]{S-T}. \end{singlespace}
\end{rem*}
\begin{thm}
\begin{singlespace}
The Bernoulli shift $\left(\bbX,\mathcal{B}\left(\bbX\right),P,T\right)$
is a non-singular, type ${\rm III}_{1}$, zero-type and power weakly
mixing transformation. \end{singlespace}
\end{thm}
\begin{proof}
\begin{singlespace}
In \cite{Kos} it is shown that the shift is a type ${\rm III}_{1}$
transformation. Therefore by Theorem \ref{thm:Let--be} the shift
is a zero type transformation. It remains to show that the shift is
power weak mixing. 

Let $l_{1},l_{2},...,l_{k}\in\bbZ\backslash\{0\}$ and denote by $S:=T^{l_{1}}\times T^{l_{2}}\times\cdots\times T^{l_{k}}$.
Clearly \begin{eqnarray*}
S^{n'}\left(w_{1},w_{2},..,w_{k}\right) & = & \prod_{i=1}^{k}T^{\left(l_{i}n\right)'}\left(w_{i}\right)\end{eqnarray*}
and $\left(\bbX^{k},P^{\times k},S\right)$ admits an exhaustive and
exact factor. Therefore in order to prove the ergodicity of $S$ it
is sufficient to show that $S$ is conservative. 

By \cite[lemma 3]{Kos} and a similar calculation for negative $n's$
there exists $t_{0}\in\bbN$ such that for every $t>t_{0}$ , $|n|\in\left[N_{t},m_{t}\right)$
and $w\in\bbX$\[
T^{n'}(w)\geq\sqrt[k]{\frac{1}{2}}\prod_{u=1}^{t}\lambda_{u}^{\sum_{j=-N_{u}+1}^{-M_{u-1}}\left\{ w_{k+n}-w_{k}\right\} }\geq2^{-N_{t}-1/k}.\]
Here the last inequality follows from $\lambda_{u}<\lambda_{1}=2$. 

Let $L=\max\left\{ \left|l_{i}\right|:1\leq i\leq k\right\} .$Then
for every $t>\max\left(t_{0},L\right)$ , $i\in\{1,..,k\}$ and $N_{t}\leq n\leq\frac{m_{t}}{L}$,
\[
T^{\left(l_{i}n\right)'}(w)\geq2^{-N_{t}-1/k}\]
and so \[
S^{n'}\left(w_{1},w_{2},..,w_{k}\right)\geq2^{-kN_{t}-1}.\]
Which together with \eqref{eq:The growth condition necessary for p.w.m}
implies that every $\tilde{w}\in\bbX^{k}$ and $t>\max\left(t_{0},k,L\right)$
\[
\sum_{n=N_{t}}^{m_{t}/L}S^{n'}(\tilde{w})\geq\left(\frac{m_{t}}{L}-N_{t}\right)2^{-kN_{t}-1}\geq\frac{1}{2}.\]
Therefore for every $ $$\tilde{w}\in\bbX^{k}$, \[
\sum_{n=1}^{\infty}S^{n'}\left(\tilde{w}\right)\geq\sum_{t}\sum_{n=N_{t}}^{m_{t}/L}S^{n'}(\tilde{w})=\infty.\]
By Hopf's theorem for non-singular transformations $S$ is conservative. \end{singlespace}

\end{proof}
\begin{singlespace}
The next example is a continuous time flow such that all the times
are zero type and power weakly mixing. In this Markov Chain example
the flow preserves an infinite measure. 
\end{singlespace}

\subsection{The Markov Chain Example.}

A Borel map $X\times\RR\ni(x,t)\mapsto\phi_{t}(x)$ such that \[
\phi_{t}\phi_{s}=\phi_{t+s}\]
 is called a \textit{non-singular flow }on $(X,\cB)$. Given a measure
$\mu$ on $X$, the semi-flow $\left\{ \phi_{t}\right\} _{t\in[0,\infty)}$
is called exact if \[
\cap_{t\geq0}\phi_{t}^{-1}\cB=\left\{ \emptyset,X\right\} mod\mu.\]

$ $

\begin{singlespace}
A measure preserving flow $\left(X,\mathcal{B},\mu,\left\{ \phi_{t}\right\} \right)$
is a \textit{K-flow} if it admits an exhaustive and exact factor.
Clearly a natural extension of an exact semiflow is a K-flow. 
\end{singlespace}

The flow is called\textit{ Power Weakly Mixing} if for every $t_{1},t_{2},..,t_{n}\in\RR$,
\[
\phi_{t_{1}}\times\phi_{t_{2}}\times\cdots\phi_{t_{n}}\]
is an ergodic transformation of $\left(\left(S^{\RR}\right)^{k},\cB^{\otimes k},\mu^{\otimes k}\right)$. 

A function $p:[0,\infty)\to[0,1]$ is a \textit{Markovian Renewal
Function} if there exists a countable state space, which will be denoted
by $S$, Markov Chain $\left\{ X_{t}\right\} _{t\in[0,\infty)}$ and
a state $a\in S$ such that \[
P_{a,a}(t):=P\left(\left.X_{t}=a\right|X_{0}=a\right)=p(t).\]

We say that $p$ is \textit{aperiodic} if \[
\gcd\left\{ n\in\mathbb{N}:\ p(n)\neq0\right\} =1,\]
and \textit{null recurrent} if \[
\sum_{n=1}^{\infty}p(n)=\infty\ and\ p(n)\xrightarrow[n\to\infty]{}0.\]
Given a renewal function $p$ the sequence $\left\{ p(n)\right\} _{n\in\mathbb{N}}$
defines a renewal sequence for the discrete time Markov Chain $\left\{ X_{n}\right\} _{n=0}^{\infty}$.
Thus if $p$ is aperiodic and null recurrent then the Markov chain
$\left\{ X_{n}\right\} _{n=0}^{\infty}$ is aperiodic and null recurrent.
Hence there exists a stationary ($\sigma$-finite) measure $\tilde{\mu}\in\mathcal{M}(S)$.
It follows that the measure \[
\mu(\{a\})=\int_{0}^{1}\left(\sum_{s\in S}P_{s,a}(t)\tilde{\mu}(\{s\})\right)dt\]
is a stationary measure for $\left(P_{t}\right)_{t\in\RR}=\left(\left\{ P_{a_{1},a_{2}}(t)\right\} _{a_{1},a_{2}\in S}\right)_{t\in\RR}$.
Finally let $\nu=P^{\mu}$ be the measure on $S^{\RR}$ with finite
dimensional distributions \[
\nu\left[x_{t_{0}}=s_{0},x_{t_{1}}=s_{1},...,x_{t_{n}}=s_{n}\right]=\mu\left(\left\{ s_{0}\right\} \right)P_{s_{0},s_{1}}\left(t_{1}-t_{0}\right)\cdots P_{s_{n-1},s_{n}}\left(t_{n}-t_{n-1}\right),\]
for every $t_{0}<t_{1}<\cdots<t_{n}$ and $s_{0},s_{1},\ldots,s_{n}\in S$.
The flow $\left\{ \phi_{t}\right\} _{t\in\mathbb{R}}$ on $S^{\RR}$
defined by\[
\phi_{t}w(s)=w(s+t)\]
 is $\nu$- measure preserving. It is the natural extension of the
semiflow $\left\{ \phi_{t}\right\} _{t\in[0,\infty)}$. 
\begin{thm}
\label{thm:Power weakly mixing flow}Let $p:[0,\infty)\to[0,1]$ be
an aperiodic and null recurrent Markov Renewal Function then the flow
$\left(S^{\RR},\mathcal{B},\nu,\left\{ \phi_{t}\right\} _{t\in\RR}\right)$
is conservative, exact and zero-type. If in addition for every $t_{1},t_{2},..,t_{n}\in\RR_{+}$,
\[
\sum_{n=1}^{\infty}\prod_{j=1}^{k}p_{a,a}^{(n)}(t_{j})=\infty\]
then the flow is Power Weakly Mixing. \end{thm}
\begin{proof}
Since $\left\{ X_{t}\right\} _{t\geq0}$ is a null recurrent Markov
chain the flow is conservative and zero-type. 

First we show that the tail $\sigma$-field of $\left\{ X_{t}\right\} _{t\geq0}$
is trivial, hence the semiflow $\left\{ \phi_{t}\right\} _{t\in[0,\infty)}$
is exact. Let $h>0$ and observe that $\Upsilon=\left(S^{h\mathbb{N}},\cB_{S^{h\mathbb{N}}},\nu|_{h\mathbb{N}},\phi_{h}\right)$
is a factor of $ $$\mathfrak{X}=\left(S^{[0,\infty)},\mathcal{B},\nu,\phi_{h}\right)$.
Since the discrete time chain $\left\{ X_{nh}\right\} _{n\in\mathbb{N}}$
is aperiodic and recurrent, it follows by \cite{BF} that its tail
$\sigma$-algebra is trivial. 

Denote by $\mathcal{F}_{h}=\mathcal{B\cap}S^{[0,h]}.$ Then it follows
from the Markov property that given $\Upsilon$, for every $\Lambda_{0},\Lambda_{1},...,\Lambda_{n}\in\mathcal{F}_{h}$,
the sets \[
\left[w\in\Lambda_{1}\right],\left[\phi_{h}w\in\Lambda_{2}\right],...,\left[\phi_{nh}w\in\Lambda_{n}\right]\]
are independent. By Kolmogorov's zero-one law $\mathfrak{X}$ is an
exact non-singular extension of $\Upsilon$ in the sense of \cite{AD}.
Therefore, since $\Upsilon$ is exact, it follows by Proposition 4
in \cite{AD} that $\mathfrak{X}$ is exact. See also \cite[Theorem 6]{Ios}.
Therefore since the flow $\left\{ \phi_{t}\right\} _{t\in\RR}$ is
the natural extension of the semiflow, it is a K-flow. 

It follows that for every $t_{1},t_{2},..,t_{k}\in\RR$ the transformation
$R=\phi_{t_{1}}\times\phi_{t_{2}}\times\cdots\phi_{t_{n}}$ is K and
in order to prove ergodicity of $R$ it is enough to show conservativity
which is a consequence of the fact that,

\begin{eqnarray*}
\sum_{n=1}^{\infty}\prod_{j=1}^{k}p_{a,a}^{(n)}\left(\left|t_{j}\right|\right) & = & \infty.\end{eqnarray*}

\end{proof}
Example: Let 

\[
p(t)=\frac{1}{\log(e+t)}.\]
Since $p(t)$ satisfies the conditions of \cite[Theorem 6.6, p.144, see also p.41]{Kin}
there exists a continuous time Markov Chain $\left\{ X_{t}:t\in\mathbb{R}\right\} $
on a countable state space $S$ such that for some $a\in S$, \[
p_{a,a}(t):=P\left(\left.X_{t}=a\right|X_{o}=a\right)=p(t).\]
 Since \[
p(n)=\frac{1}{\log(e+n)}\neq0,\ \sum_{n=1}^{\infty}p(n)=\infty\]
and for every $t_{1},t_{2},..,t_{n}>0$, \[
\sum_{n=1}^{\infty}\prod_{j=1}^{k}p_{a,a}^{(n)}(t_{j})=\sum_{n=1}^{\infty}\prod_{j=1}^{k}p_{a,a}(n\cdot t_{j})=\sum_{n=1}^{\infty}\prod_{j=1}^{k}\frac{1}{\log\left(e+t_{j}\cdot n\right)}=\infty,\]
it satisfies the conditions of Theorem \ref{thm:Power weakly mixing flow}
and hence the Markov flow defined by the Markov Chain is conservative,
zero type and power weakly mixing. 

\begin{singlespace}
\textit{\textcolor{black}{Acknowledgments.}}\textcolor{red}{ }\textcolor{black}{The
author would like to thank his advisor Prof. Jon Aaronson for his
time, patience and advice. }
\end{singlespace}

\end{document}